\documentclass[12pt]{article}
\usepackage{array,amsmath,amssymb,MnSymbol}
\usepackage{fullpage,lineno}
\usepackage{color}
\def \red {\textcolor{red} }
\renewcommand{\baselinestretch}{1.15}		

\begin{document}


\input epsf.tex
\def\sp{\bigskip}
\def\ti{\\ \hglue \the \parindent}
\def\gjoin{\diamondplus}

\newtheorem{theorem}{Theorem}[section]
\newtheorem{lemma}[theorem]{Lemma}
\newtheorem{corollary}[theorem]{Corollary}
\newtheorem{prop}[theorem]{Proposition}
\newtheorem{defn}[theorem]{Definition}
\newtheorem{conj}[theorem]{Conjecture}
\newtheorem{claim}[theorem]{Claim}
\newtheorem{remark}[theorem]{Remark}
\newtheorem{alg}[theorem]{Algorithm}
\def\qed{\ifhmode\unskip\nobreak\hfill$\Box$\bigskip\fi \ifmmode\eqno{Box}\fi}
\newenvironment{proof}{\par\noindent{\bf Proof.} }{\qed\par\smallskip}

\def\al{\alpha} \def\be{\beta}  \def\ga{\gamma} \def\dlt{\delta}
\def\eps{\epsilon} \def\th{\theta}  \def\ka{\kappa} \def\lmb{\lambda}
\def\sg{\sigma} \def\om{\omega}
\def\nul{\varnothing} 
\def\st{\colon\,}   
\def\esub{\subseteq}
\def\VEC#1#2#3{#1_{#2},\ldots,#1_{#3}}
\def\VECOP#1#2#3#4{#1_{#2}#4\cdots #4 #1_{#3}}
\def\SE#1#2#3{\sum_{#1=#2}^{#3}}  \def\SGE#1#2{\sum_{#1\ge#2}}
\def\PE#1#2#3{\prod_{#1=#2}^{#3}} \def\PGE#1#2{\prod_{#1\ge#2}}
\def\UE#1#2#3{\bigcup_{#1=#2}^{#3}}
\def\SM#1#2{\sum_{#1\in #2}}
\def\CH#1#2{\binom{#1}{#2}} \def\MULT#1#2#3{\binom{#1}{#2,\ldots,#3}}
\def\FR#1#2{\frac{#1}{#2}}
\def\FL#1{\left\lfloor{#1}\right\rfloor} \def\FFR#1#2{\FL{\frac{#1}{#2}}}
\def\CL#1{\left\lceil{#1}\right\rceil}   \def\CFR#1#2{\CL{\frac{#1}{#2}}}
\def\Gb{\overline{G}}
\def\un#1{\underline{#1}}

\def\B#1{{\bf #1}}      \def\R#1{{\rm #1}}
\def\I#1{{\it #1}}      \def\c#1{{\cal #1}}
\def\C#1{\left | #1 \right |}    
\def\dfc{\R{def}}
\def\CC#1{\Vert #1\Vert}
\def\hQ{\hat Q}
\def\hR{\hat R}
\def\hS{\hat S}
\def\hT{\hat T}

\title{Another Proof of the Generalized Tutte--Berge Formula for $f$-Bounded
Subgraphs}

\author{ 
Zishen Qu\thanks{University of Illinois; {\tt zishenq2@illinois.edu.}
Research supported by Natural Sciences and Engineering Research Council of
Canada (NSERC), [funding reference number PGSD-568936-2022];
financ\'ee par le Conseil de recherches en sciences naturelles
et en gnie du Canada (CRSNG), [num\'ero de r\'ef\'erence PGSD-568936-2022]}
\ and Douglas B. West\thanks{Zhejiang Normal University and University of
Illinois; {\tt dwest@illinois.edu.}  Research supported by National Natural
Science Foundation of China grants NSFC 11871439, 11971439, and U20A2068.}
}
\date{\today}
\maketitle

\begin{abstract}
Given a nonnegative integer weight $f(v)$ for each vertex $v$ in a multigraph
$G$, an {\it $f$-bounded subgraph} of $G$ is a multigraph $H$ contained in $G$
such that $d_H(v)\le f(v)$ for all $v\in V(G)$.  Using Tutte's $f$-Factor
Theorem, we give a new proof of the min-max relation for the maximum size of an
$f$-bounded subgraph of $G$.  When $f(v)=1$ for all $v$, the formula reduces to
the classical Tutte--Berge Formula for the maximum size of a matching.
\end{abstract}

\section{Introduction}
Let $f$ assign a nonnegative integer weight to each vertex $v$ in
a multigraph $G$.  An {\it $f$-factor} of $G$ is a spanning subgraph $H$ of $G$
such that $d_H(v)=f(v)$ for all $v\in V(G)$, where $d_H(v)$ denotes the degree
of vertex $v$ in $H$.  The famous $f$-Factor Theorem of
Tutte~\cite{Tutf1,Tutf2} characterizes when an $f$-factor exists in $G$.
When $f(v)=1$ for all $v$, the condition includes the earlier Tutte Condition
(Tutte~\cite{Tut1}) that is necessary and sufficient for the existence of
$\FR12\C{V(G)}$ pairwise disjoint edges in $G$.  When this condition fails, the
extent of failure provides a formula for the maximum number of pairwise
disjoint edges in $G$, called the Tutte--Berge Formula (Berge~\cite{Ber}).
The analogous behavior holds for the $f$-Factor Theorem.

An {\it $f$-bounded subgraph} of $G$ is a spanning subgraph $H$ such that
$d_H(v)\le f(v)$ for all $v\in V(G)$.  The number of edges in an $f$-factor
must be $\FR12 f(V(G))$, where any function $g$ on $V(G)$ extends to subsets
$S\esub V(G)$ via $g(S)=\SM vS g(v)$.  The maximum number of edges in an
$f$-bounded subgraph of an $n$-vertex multigraph $G$ is $\FR12(f(V(G))-\gamma)$,
where $\gamma$ measures the maximum failure of the $f$-factor condition on $G$
as described below.  Given disjoint vertex sets $A$ and $B$, let $\CC{A,B}$
denote the number of edges having endpoints in both $A$ and $B$, and let
$\CC{A}$ be the number of edges in the subgraph $G[A]$ induced by $A$.

\begin{theorem}[\R{Tutte~\cite{Tutf1,Tutf2}}]\label{ffac}
Given a multigraph $G$ with nonnegative weight function $f$ on $V(G)$ and
disjoint sets $S,T\esub V(G)$, a component of $G-S-T$ with vertex set $Q$ is
said to be ``bad'' if $f(Q)+\CC{Q,T}$ is odd.  The multigraph $G$ has an
$f$-factor if and only if
\begin{equation*}
f(T)\le f(S)+d_{G-S}(T)-q(S,T)
\end{equation*}
for all disjoint $S,T\subset V(G)$, where $q(S,T)$ is the number of
bad components of $G-S-T$.
\end{theorem}

Necessity of the condition is easy to explain.  The contributions to $f(T)$
by an $f$-factor $H$ in $G$ come from edges incident to $S$ (at most $f(S)$)
and edges not incident to $S$ (at most $d_{G-S}(T)$).  However, for a component
$G[Q]$ of $G-S-T$, the subgraph of $H$ consisting of edges incident to $Q$ must
have even degree-sum, so the degree-sum cannot equal $f(Q)+\CC{Q,T}$ if that
quantity is odd.  In that case, since omitting edges induced by $Q$ does not
change the parity, $H$ must omit some edge joining $Q$ to $T$ or use some edge
joining $Q$ to $S$, which steals $1$ from the contribution of $d_{G-S}(T)$ or
$f(S)$, respectively, to the bound on $f(T)$.

When $f(v)=1$ for all $v$, the edges of an $f$-bounded subgraph are disjoint
and form a {\it matching} in $G$.  For a pair $(S,T)$ with $T=\nul$,
a component of $G-S$ with vertex set $Q$ is bad precisely when $\C Q$ is odd;
a $1$-factor must match some vertex of $Q$ to a vertex of $S$.  The restriction
of the $f$-factor condition to the case $T=\nul$ then reduces to
$o(G-S)\le\C S$, where $o(H)$ is the number of components of $H$ having an odd
number of vertices.  Tutte's $1$-Factor Theorem~\cite{Tut1} asserts that this
condition is sufficient.  To obtain this as the special case of the $f$-Factor
Theorem for $f(v)=1$ for all $v\in V(G)$, one would need to show that the truth
of the condition for all $S$ with $T=\nul$ implies it also when $T\ne\nul$.

In the matching case, it is clear that $o(G-S)-\C S$ vertices must remain
unmatched when $o(G-S)>\C S$.  The Tutte--Berge Formula says that the strongest
such bound is sharp.

\begin{theorem}[\R{Berge~\cite{Ber}}]\label{btut}
In a multigraph $G$, the maximum number of edges in a matching
is $\FR12[n-\max_{S\esub V(G)}(o(G-S)-\C S)]$.
\end{theorem}

Schrijver~\cite{Sch} extended this formula to general $f$, showing that the
extent of failure of the $f$-factor condition in Theorem~\ref{ffac} determines
the maximum size of an $f$-bounded subgraph.  Define the {\it deficiency}
$\dfc(S,T)$ of an ordered pair $(S,T)$ of disjoint vertex subsets in a
multigraph $G$ by
\begin{equation}\label{ddfc}
\dfc(S,T)=f(T)-f(S)-d_{G-S}(T)+q(S,T).
\end{equation}
In this language, Schrijver's result is the following.

\begin{theorem}\label{btgen}
Given a nonnegative weight function $f$ on a multigraph $G$,
the maximum number of edges in an $f$-bounded subgraph of $G$ is 
\begin{equation*}
\FR12[f(V(G))-\max_{S,T\esub V(G)}\dfc(S,T)],
\end{equation*}
where the maximization is over ordered pairs $(S,T)$ of disjoint vertex subsets
in $G$.
\end{theorem}

Consistent with the linear programming approach to matching, Schrijver used the
notation $b$ for our $f$ and used the term ``simple $b$-matching'' for our
``$f$-bounded subgraph''.  In our notation, his expression for the result
was as follows.

\begin{theorem}[{\rm Schrijver~\cite{Sch}, p.~569}]\label{schr}
Given a vertex weighting $f$ on a multigraph $G$, the maximum size of an
$f$-bounded subgraph is the minimum, over all pairs $(S,T)$ of disjoint subsets
of $V(G)$, of
\begin{equation*}
f(S)+\CC{T}+\sum\FL{(f(Q)+\CC{Q,T})/2},
\end{equation*}
where the sum is over components $G[Q]$ of $G-S-T$.
\end{theorem}

To see that our Theorem~\ref{btgen} involving Tutte's $f$-factor deficiency is
the same as this theorem by Schrijver, we let $R=V(G)-S-T$ and compute
\begin{align*}
\FR12[f(V(G))-\dfc(S,T)]&=
\FR12[f(R)+f(S)+f(T)-f(T)+f(S)+d_{G-S}(T)-q(S,T)]\\
&=f(S)+\FR12[f(R)+2\CC{T}+\CC{T,R}-q(S,T)]\\
&=f(S)+\CC{T}+\sum\FL{(f(Q)+\CC{Q,T})/2}.
\end{align*}

Schrijver obtained his result by setting all capacities to $1$ in a more
general theorem where edges have capacities.  That theorem he reduced to another
min-max relation for $f$-bounded subgraphs, which he proved by expanding each
vertex $v$ into a set of size $f(v)$ and each edge into a complete bipartite
graph and then applying the Tutte--Berge Formula.  He observed that that
approach also gives an independent proof of Tutte's $f$-Factor Theorem,
which is our starting point.

An $f$-factor with $f(v)=k$ for all $v$ is called a ``$k$-factor''; similarly
we refer to a $k$-bounded subgraph.  In addition to the Tutte--Berge Formula
for maximum $1$-bounded subgraphs, maximum $2$-bounded subgraphs have been
studied, under the name ``simple $2$-matching''.  For simple subcubic graphs,
Hartvigsen and Li~\cite{HL} provided a polynomial-time algorithm to find a
maximum $2$-bounded subgraph and a structural theorem for the resulting
subgraphs.  They also observed that in the min-max relation on that class
(in our notation) one need only optimize over pairs $(S,T)$ such that $G[T]$
has no edges.  This is analogous to restricting to pairs $(S,T)$ with $T=\nul$
in the Tutte--Berge Formula.

\section{Weak Duality and Parity}
Let $h(G;f)$ be the maximum sum of vertex degrees of an $f$-bounded subgraph
of $G$.  We use $h$ to avoid the factor of $1/2$ in the statements.  To
streamline notation in applying $f$ to the full vertex set, let $U=V(G)$.

We want to prove $h(G;f)=\min_{S,T}[f(U)-\dfc(S,T)]$.  As in most min-max
relations, ``weak'' duality is easy to prove.  It is a bit more subtle than the
necessity of the
$f$-factor condition because we cannot assume $d_H(Q)=f(Q)$ when $H$ is 
an $f$-bounded subgraph of $G$ and $G[Q]$ is a component of $G-S-T$.


\begin{lemma}\label{weak}
For a multigraph $G$ with vertex set $U$ and disjoint $S,T\esub U$, every
$f$-bounded subgraph of $G$ has degree-sum at most $f(U)-\dfc(S,T)$.
\end{lemma}
\begin{proof}
Given an $f$-bounded subgraph $H$ of $G$, we must prove
$$d_H(U)\le f(U)-f(T)+f(S)+d_{G-S}(T)-q(S,T).$$
Since $d_H(S)\le f(S)$, it suffices to prove
$$d_H(T)+d_H(R)\le f(S)+f(R)+d_{G-S}(T)-q(S,T),$$
where $R=U-S-T$.  The contributions to $d_H(T)+d_H(R)$ using edges incident
to $S$ sum to at most $f(S)$, and the contributions using edges
within $T$ sum to at most $2\CC{T}$.

The remaining edges of $H$ are within $R$ or join $T$ to $R$.  With $Q$
denoting the vertex set of a component of $G[R]$, let $H'$ be the subgraph of
$H$ consisting of $H[Q]$ plus the edges in $H$ joining $Q$ to $T$.  The
degree-sum of this subgraph is at most $f(Q)+\CC{Q,T}$.  However, since $H'$
is a graph, its degree sum is even, so when the component is bad we must
subtract $1$ from this bound on its degree sum.  Summing over all components of
$G[R]$ yields the bound $f(R)+d_{G-S}(T)-q(S,T)$ on the contributions to
$d_H(T)+d_H(R)$ by edges not incident to $S$, since
$d_{G-S}(T)=2\CC{T}+\CC{R,T}$.
\end{proof}

The special case $\dfc(\nul,\nul)$ always evaluates to the number of components
of $G$ on which $f$ sums to an odd value.  Hence the maximum deficiency is
always nonnegative.  We also have a parity condition, which is well-known
and is consistent with the factor of $1/2$ in Theorem~\ref{btgen}.

\begin{lemma}[Parity Lemma]\label{parity}
For a multigraph $G$ with vertex set $U$ and disjoint $S,T\esub U$, the
values of $\dfc(S,T)$ and $f(U)$ have the same parity.
\end{lemma}
\begin{proof}
It suffices to show that $f(U)-\dfc(S,T)$ is even.  Canceling even quantities,
including $2\CC{T}$, this value has the same parity as
$f(R)+\CC{R,T}-q(S,T)$.  Splitting up $f(R)+\CC{R,T}$ by the components
of $G[R]$, the parity of $f(R)+\CC{R,T}$ is the same as the parity of 
$q(S,T)$, since $q(S,T)$ counts $1$ for each component $G[Q]$ of $G[R]$
such that $f(Q)+\CC{Q,T}$ is odd.  Hence the specified quantity is even.
\end{proof}

\section{Proof of the Theorem}

To prove the min-max relation, we must prove that the upper bound in
Lemma~\ref{weak} is sharp.  We apply the $f$-Factor Theorem, using a multigraph
analogue of the trick used in proving the Tutte--Berge Formula from Tutte's
$1$-Factor Theorem.

\begin{theorem}
For a multigraph $G$ with vertex set $U$ and a nonnegative weight function $f$,
$$
h(G;f)=\min_{S,T}[f(U)-\dfc(S,T)],
$$
where the minimum is over all disjoint subsets $S,T\esub U$.
\end{theorem}

\begin{proof}
Let $\gamma=\max_{S,T}\dfc(S,T)$.  As we have noted,
$\gamma\ge\dfc(\nul,\nul)\ge0$.
Lemma~\ref{weak} proves $h(G;f)\le f(U)-\gamma$; we prove that $G$ has
an $f$-bounded subgraph with degree-sum $f(U)-\gamma$.
Tutte's $f$-Factor Theorem gives us this result when $\gamma=0$,
so we may assume $\gamma\ge1$.

Form $G'$ by adding to $G$ a single vertex $w$ having $\gamma$ edges to each
vertex of $G$, and define $f'$ on $V(G')$ by letting $f'(w)=\gamma$ and
$f'(v)=f(v)$ for $v\ne w$.  The graph $G$ has an $f$-bounded subgraph with
degree-sum $f(U)-\gamma$ if and only if $G'$ has an $f'$-factor.  Hence it
suffices to show that $G'$ satisfies Tutte's condition for an $f'$-factor,
given $\dfc(S,T)\le\gamma$ for all $S,T\esub U$.

Let $U'=U\cup\{w\}$.  We consider an arbitrary partition $(R',S',T')$ of $U'$,
from which we define a corresponding partition $(R,S,T)$ of $U$ by deleting $w$
from the part in which it appears; that is, $X=X'\cap U$, where $X\in\{R,S,T\}$.

Let $\rho'=f'(S')+d_{G'-S'}(T')-q'(S',T')$ and $\rho=f(S)+d_{G-S}(T)-q(S,T)$,
where $q'(S',T')$ counts the components $G'[Q']$ of $G'[R']$ such that
$f'(Q')+\CC{Q',T'}$ is odd.  By definition, the deficiency of $(S',T')$
in $G'$ under $f'$ is $f'(T')-\rho'$, and we need to guarantee that this
deficiency is nonpositive.  Thus what we need to prove is $\rho'\ge f'(T')$.

By the definition of $\gamma$, we have $\rho\ge f(T)-\gamma$ for disjoint
$S,T\in V(G)$.  We consider three cases, depending on which part of
$(R',S',T')$ contains $w$.

\smallskip
{\bf Case 1:} {\it $w\in S'$.}
Here $f'(S')=f(S)+\gamma$, $d_{G'-S'}(T')=d_{G-S}(T)$, and $q'(S',T')=q(S,T)$.
Thus $\rho'=\rho+\gamma\ge f(T)=f'(T')$.

\smallskip
{\bf Case 2:} {\it $w\in R'$.}
Here $f'(S')=f(S)$ and $f'(T')=f(T)$.  Since $w\in V(G')-S'$, we have
$d_{G'-S'}(T')=d_{G-S}(T)+\gamma\C T$.  The graph $G'[R']$ is connected,
so $q'(S',T')\in\{0,1\}$.  Using $\rho\ge f(T)-\gamma$, we have
\begin{align}\label{inR}
\rho'&=f(S)+d_{G-S}(T)+\gamma\C T-q'(S',T')\nonumber\\
&=\rho+q(S,T)+\gamma\C T-q'(S',T')
~\ge~ f(T)+\gamma(\C T-1)+q(S,T)-q'(S',T')\nonumber\\
&= f'(T')+\gamma(\C T-1)-q'(S',T')+q(S,T),
\end{align}
Since $\gamma\ge1$ and $q'(S',T')\le 1$, the final expression above is at least
$f'(T')$ when $\C T\ge2$.

When $T=\nul$, we have $f'(T')=f(T)=0$, so we want $\rho'\ge0$.
From the first line of~\eqref{inR}, $\rho'=f(S)-q'(S',T')$.
Note that $q'(S',T')$, which is $0$ or $1$, has the same parity as $f'(R')$
since $\CC{R,T}=0$.  Thus $\rho'=f(S)$ if $f'(R')$ is even, while
$\rho'=f(S)-1$ if $f'(R')$ is odd.  Since $f'(S')\ge0$, proving $\rho'\ge0$
now follows from $f'(S')$ and $f'(R')$ having the same parity.  This holds
because $f'(U')$ is even, which follows from $f'(U')=f(U)+\gamma$ by the Parity
Lemma.

Finally, suppose $\C T=1$.  If $q'(S',T')=0$, then the computation in 
\eqref{inR} yields $\rho'\ge f'(T')$, so we may assume $q'(S',T')=1$.
Thus $f'(R')+\CC{R',T'}$, which equals $f(R)+2\gamma+\CC{R,T}$ because
$\C T=1$, is odd.  In particular, $f(R)+\CC{R,T}$ is odd.  This quantity
is the sum of the quantities of the form $f(Q)+\CC{Q,T}$ for the components
$G[Q]$ of $G-S-T$.  Hence at least one of them is odd, so $q(S,T)\ge1$,
which yields $\rho'\ge f'(T')$ by \eqref{inR}.

\smallskip
{\bf Case 3:} {\it $w\in T'$.}
Here $f'(S')=f(S)$ and $f'(T')=f(T)+\gamma$, but
$$d_{G'-S'}(T')=d_{G-S}(T)+\gamma(2\C T+\C R).$$
For $q'(S',T')$, we count components $G'[Q]$ of $G'-S'-T'$ such that
$f(Q)+\CC{Q,T}+\gamma\C Q$ is odd.  There are at most
$\C R$ such components.  We also use $\rho\ge f(T)-\gamma$ and $\gamma\ge1$
to compute
\begin{align*}
\rho'
&=f'(S')+d_{G'-S'}(T')-q'(S',T')\\
&\ge f(S)+d_{G-S}(T)+2\gamma\C T+(\gamma-1)\C R
~=~\rho+q(S,T)+2\gamma\C T+(\gamma-1)\C R\\
&\ge f(T)+\gamma(2\C T-1)+q(S,T)
~\ge~ f'(T')+\gamma(2\C T-2)+q(S,T).
\end{align*}
If $\C T\ge1$, then we obtain $\rho'\ge f'(T')$, as desired.

Hence we may assume $T=\nul$, so $f'(T')=\gamma$ and
$d_{G'-S'}(T')=\gamma\C R$.  We want $\rho'\ge\gamma$.
Since $q'(S',T')\le\C R$, we now have
\begin{align}\label{inT0}
\rho'
&=f'(S')+d_{G'-S'}(T')-q'(S',T')
=f(S)+\gamma\C R-q'(S',T')
\ge f(S)+(\gamma-1)\C R .
\end{align}
Since $f(S)\ge0$, we obtain $\rho'\ge (\gamma-1)\C R\ge 2\gamma-2\ge \gamma$
if $\gamma$ and $\C R$ are both at least $2$.

When $\gamma=1$, we seek $\rho'\ge1$, and \eqref{inT0} suffices unless
$f(S)=0$ and $q'(S',T')=\C R$.  The latter requires
$f'(v)+\CC{\{v\},T'}$ odd for each $v\in R$.  Since $T'=\{w\}$ and $\gamma=1$,
each $\CC{\{v\},T'}$ is $1$, so $f(v)$ is even for each $v\in R$.  Also,
$$
f'(U')=f'(S')+f'(T')+f'(R')=1+f'(R).
$$
Since $f'(U')$ is even by the Parity Lemma, $f(v)$ must be odd for some
$v\in R$.  The contradiction finishes the case $\gamma=1$.

We are left with $\C R\le 1$ and $\gamma\ge2$.
When $\C R=1$, with $R=\{v\}$, \eqref{inT0} reduces to
$\rho'=f(S)+\gamma-q'(S',T')$, with $q'(S',T')\in\{0,1\}$.
This yields $\rho'\ge\gamma$ unless $f(S)=0$ and $q'(S',T')=1$.
Since $\CC{\{v\},T'}=\gamma$, the latter requires
$f(v)+\gamma$ to be odd.  However, since $f'(T')=f'(w)=\gamma$,
$$f'(U')=f'(R')+f'(S')+f'(T')=f(v)+f(S)+\gamma=f(v)+\gamma,$$
which contradicts that $f'(U')$ is even.

Finally, we have $R=T=\nul$ with $\gamma\ge2$ and $S=U$.
Since $\C{R'}=0$, also $q'(S',T')=0$.  Thus $\rho'=f'(S)=f(U)$,
and we want $f(U)\ge\gamma$.  For any disjoint $\hS,\hT\esub U$,
let $\hR=U-\hS-\hT$.  The value $q(\hS,\hT)$ counts $1$ for each
component $G[\hQ]$ of $G[\hR]$ such that $f(\hQ)+\CC{\hQ,\hT}$ is odd.
When the value is odd it is positive, and when it is even it is nonnegative,
so we obtain an upper bound on $\dfc(\hS,\hT)$ if we count the full
value $f(\hQ)+\CC{\hQ,\hT}$ for each component.  Now we compute
\begin{align*}
\dfc(\hS,\hT)
&=q(\hS,\hT)-d_{G-\hS}(\hT)+f(\hT)-f(\hS)\\
&\le f(\hR)+\CC{\hR,\hT}-d_{G-\hS}(\hT)+f(\hT)-f(\hS)\\
&\le f(\hR)+d_{G-\hS}(\hT)-d_{G-\hS}(\hT)+f(\hT)+f(\hS) ~=~ f(U)
\end{align*}
Since the deficiency of each pair is at most $f(U)$, the maximum
deficiency is at most $f(U)$; that is, $f(U)\ge\gamma$.
\end{proof}


\begin{thebibliography}{99}
{\small
\renewcommand{\baselinestretch}{.92}            
\frenchspacing
\parskip=.5ex

\bibitem{Ber}
C. Berge, Sur le couplage maximum d'un graphe (French).
{\it C. R. Acad. Sci. Paris} 247 (1958), 258--259.

\bibitem{HL}
D. Hartvigsen and Y. Li, Maximum cardinality simple 2-matchings in
subcubic graphs.
{\it SIAM J. Optim.} 21 (2011), 1027--1045.

\bibitem{Sch}
A. Schrijver, {\it Combinatorial optimization. Polyhedra and efficiency.
Vol. A. Paths, flows, matchings.} Chapters 1--38. Algorithms and
Combinatorics Vol. 24(A).  (Springer-Verlag, 2003). 

\bibitem{Tut1}
W. T. Tutte, The factorization of linear graphs.
{\it J. London Math.} Soc. 22 (1947), 107--111.

\bibitem{Tutf1}
W. T. Tutte, The factors of graphs. 
{\it Canad. J. Math.} 4 (1952), 314--328.

\bibitem{Tutf2}
W. T. Tutte, A short proof of the factor theorem for finite graphs.
{\it Canad. J. Math.} 6 (1954), 347--352.
}

\end{thebibliography}
\end{document}